\documentclass{amsart}

\usepackage{microtype}
\usepackage[utf8]{inputenc} 
\usepackage[T1]{fontenc} 
\usepackage[english]{babel}
\usepackage{xcolor}
\usepackage{lmodern}
\usepackage{xfrac}
\usepackage{tikz}
\usepackage{pgfplots}
\pgfplotsset{compat=newest}
\usepgfplotslibrary{fillbetween}

\theoremstyle{plain} 
\newtheorem{theorem}{Theorem} 
\newtheorem{lemma}[theorem]{Lemma}
\newtheorem{corollary}[theorem]{Corollary}

\theoremstyle{definition}
\newtheorem{example}[theorem]{Example}

\theoremstyle{remark}
\newtheorem{question}{Question}

\DeclareMathOperator{\mre}{Re}

\begin{document} 
\title{F. Wiener's trick and an extremal problem for $H^p$} 
\date{\today} 

\author{Ole Fredrik Brevig} 
\address{Department of Mathematics, University of Oslo, 0851 Oslo, Norway} 
\email{obrevig@math.uio.no}

\author{Sigrid Grepstad} 
\address{Department of Mathematical Sciences, Norwegian University of Science and Technology (NTNU), NO-7491 Trondheim, Norway} 
\email{sigrid.grepstad@ntnu.no}

\author{Sarah May Instanes} 
\address{Department of Mathematical Sciences, Norwegian University of Science and Technology (NTNU), NO-7491 Trondheim, Norway} 
\email{sarahmin@stud.ntnu.no}

\begin{abstract}
	For $0<p \leq \infty$, let $H^p$ denote the classical Hardy space of the unit disc. We consider the extremal problem of maximizing the modulus of the $k$th Taylor coefficient of a function $f \in H^p$ which satisfies $\|f\|_{H^p}\leq1$ and $f(0)=t$ for some $0 \leq t \leq 1$. In particular, we provide a complete solution to this problem for $k=1$ and $0<p<1$. We also study F.~Wiener's trick, which plays a crucial role in various coefficient-related extremal problems for Hardy spaces.
\end{abstract}
\thanks{Sigrid Grepstad is supported by Grant 275113 of the Research Council of Norway. Sarah May Instanes is supported by the Olav Thon Foundation through the StudForsk program.}

\maketitle

\section{Introduction}
Let $H^p$ denote the classical Hardy space of analytic functions in the unit disc $\mathbb{D} = \left\{z\in\mathbb{C}\,:\,|z|<1\right\}$. Suppose that $k$ is a positive integer. For $0<p \leq\infty$ and $0 \leq t \leq 1$, consider the extremal problem
\begin{equation} \label{eq:Phi}
	\Phi_k(p,t) = \sup \left\{\mre\frac{f^{(k)}(0)}{k!}\,:\, \|f\|_{H^p}\leq 1 \,\text{ and }\, f(0)=t\,\right\}.
\end{equation}
By a standard normal families argument, there are extremals $f \in H^p$ attaining the supremum in \eqref{eq:Phi} for every $k\geq1$ and every $0 \leq t \leq 1$. A general framework for a class of extremal problems for $H^p$ which includes \eqref{eq:Phi} has been developed by Havinson~\cite{Havinson51}, Kabaila~\cite{Kabaila60}, Macintyre--Rogosinski~\cite{MR50} and Rogosinski--Shapiro~\cite{RS53}. A particular consequence of this theory is that the structure of the extremals is well-known (see Lemma~\ref{lem:structure} below). 

For our extremal problem, it can be deduced directly from Parseval's identity that $\Phi_k(2,t)=\sqrt{1-t^2}$ and that the unique extremal is $f(z)=t+\sqrt{1-t^2}\,z^k$. Similarly, the Schwarz--Pick inequality (see e.g.~\cite[VII.17.3]{Sarason07}) shows that $\Phi_1(\infty,t)=1-t^2$ and that the unique extremal is $f(z)=(t+z)/(1+tz)$. This served as the starting point for Beneteau and Korenblum \cite{BK04}, who studied the extremal problem \eqref{eq:Phi} in the range $1 \leq p \leq \infty$. We will enunciate their results in Section~\ref{sec:Phi0inf} and Section~\ref{sec:Phi1inf}, but for now we present a brief account of their approach.

The first step in \cite{BK04} is to compute $\Phi_1(p,t)$ and identify an extremal function. This is achieved by interpolating between the two cases $p=2$ and $p=\infty$ mentioned above, facilitated by the inner-outer factorization of $H^p$ functions. It follows from the argument that the extremal function thusly obtained is unique. 

The second step in \cite{BK04} is to show that $\Phi_k(p,t)=\Phi_1(p,t)$ for every $k\geq2$ using a trick attributed to F.~Wiener \cite{Bohr14}, which we shall now recall. Set $\omega_k=\exp(2\pi i/k)$ and suppose that $f(z)=\sum_{n\geq0} a_n z^n$. F.~Wiener's trick is based on the transform
\begin{equation} \label{eq:fwienerdef}
	W_k f(z) = \frac{1}{k}\sum_{j=0}^{k-1} f(\omega_k^j z) = \sum_{n=0}^\infty a_{kn}z^{kn}.
\end{equation} 
The triangle inequality yields that $\|W_k f\|_{H^p} \leq \|f\|_{H^p}$ for $f \in H^p$ if $1 \leq p \leq \infty$. Hence, if $f_1$ is an extremal function for $\Phi_1(p,t)$, then $f_k(z)=f_1(z^k)$ is an extremal function for $\Phi_k(p,t)$ and consequently $\Phi_k(p,t)=\Phi_1(p,t)$. Note that this argument does not guarantee that the extremal $f_k$ is unique for $\Phi_k(p,t)$.

We are interested in the extremal problem \eqref{eq:Phi} for $0<p<1$ and whether the extremal identified using F.~Wiener's trick above for $1 \leq p \leq \infty$ is unique. We shall obtain the following general result, which may be of independent interest.

\begin{theorem} \label{thm:fwiener}
	Fix $k\geq2$ and suppose that $0<p\leq \infty$. Let $W_k$ denote the F.~Wiener transform \eqref{eq:fwienerdef}. The inequality
	\[\|W_k f\|_{H^p} \leq \max\big(k^{1/p-1},1\big) \|f\|_{H^p}\]
	is sharp. Moreover, equality is attained if and only if
	\begin{enumerate}
		\item[(a)] $f \equiv 0$ when $0 <p<1$,
		\item[(b)] $W_k f = f$ when $1<p<\infty$.
	\end{enumerate}
\end{theorem}
The upper bound in the estimate is easily deduced from the triangle inequality. Hence, the novelty of Theorem~\ref{thm:fwiener} is that the inequality is sharp for $0<p<1$, and the statements (a) and (b). In Section~\ref{sec:fwiener}, we also present examples of functions in $H^1$ and $H^\infty$ which attain equality in Theorem~\ref{thm:fwiener}, but for which $W_k f \neq f$. However, we will conversely establish that if both $f$ and $W_k f$ are inner functions, then $f = W_k f$.

To illustrate the role played by the F.~Wiener transform in various coefficient related extremal problems, we first recall that the estimate $\|W_k f\|_\infty \leq \|f\|_\infty$ was originally used by F.~Wiener to resolve a problem posed by H.~Bohr \cite{Bohr14} and compute the so-called Bohr radius for $H^\infty$. We also know from \cite[Sec.~1.7]{MSUTV15} that the Krzy\.{z} conjecture on the maximal magnitude of the $k$th coefficient in the power series expansion of a non-vanishing function with $\|f\|_\infty=1$ is equivalent to the assertion that if $f$ is an extremal for the corresponding extremal problem, then $f = W_k f$. As far as we are aware, the Krzy\.{z} conjecture remains open for $k\geq6$.

Theorem~\ref{thm:fwiener} shows that the extremal for $\Phi_k(p,t)$ is unique when $1<p<\infty$. We shall see in Section~\ref{sec:Phi1inf} that the extremal problem $\Phi_k(p,t)$ with $k\geq2$ and $1 \leq p \leq \infty$ has a unique extremal except for when $p=1$ and $0\leq t < 1/2$.

In the range $0<p<1$ with $k=1$, the extremal problem \eqref{eq:Phi} has been studied by Connelly \cite[Sec.~4]{Con17}, who resolved the problem in the cases $0 \leq t < 2^{-1/p}$ and $2^{-1/p}\sqrt{p}(2-p)^{1/p-1/2}<t\leq1$. Connelly also states conjectures on the behavior of $\Phi_1(p,t)$ in the range $2^{-1/p} \leq t \leq 2^{-1/p}\sqrt{p}(2-p)^{1/p-1/2}$. The conjectures are based on numerical analysis (see \cite[Sec.~5]{Con17}).

In Section~\ref{sec:Phi0inf}, we will extend Connelly's result to the full range $0 \leq t \leq1$. Our result demonstrates that for each $0<p<1$ there is a unique $0<t_p<1/2$ such that the extremal for $\Phi_1(p,t_p)$ is not unique, thereby confirming the above-mentioned conjectures.

\begin{figure}
        \centering
        \begin{tikzpicture}
                \begin{axis}
                        [axis equal image,
                        axis lines=middle,
                        axis line style=thin,
                        axis on top,
                        xmin=0,
                        xmax=1.1,
						xtick={0.25,0.5,0.75,1},
						xticklabels={\sfrac{1}{4},\sfrac{1}{2},\sfrac{3}{4},1},
                        ymin=0,
                        ymax=1.43,
                        ytick={0.5,1,1.299},
                        yticklabels={\sfrac{1}{2},1,\sfrac{3${\sqrt{3}}\,$}{4},},
                        ylabel={$\Phi_1(p,t)$},
                        xlabel=$t$,
                        every axis x label/.style={ at={(ticklabel* cs:1.025)}, anchor=west,},
                        every axis y label/.style={ at={(ticklabel* cs:1.025)}, anchor=south,},
                        axis line style={->}, x post scale = 2.31374]
						\addplot[thin, densely dotted, color=blue, domain=0:0.5652, samples=500] {1.299};
						\addplot[thin, color=red, domain=0:1, samples=500] {1-x^2}; 
						\addplot[thin, color=orange, domain=0:1, samples=500] {sqrt(1-x^2)}; 
						\addplot[thin, color=green, domain=0:0.5, samples=250] {1}; 
						\addplot[thin, color=green, domain=0.5:1, samples=250] {2*sqrt(x-x^2)}; 
                        \input{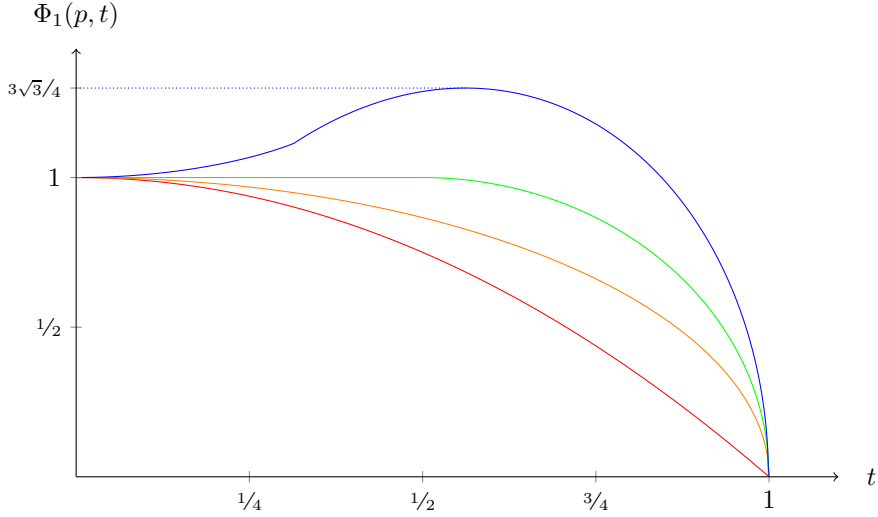}
                \end{axis}
        \end{tikzpicture}
		\caption{Plot of the curves $t\mapsto \Phi_1(p,t)$ for {\color{blue}$p=1/2$}, {\color{green}$p=1$}, {\color{orange}$p=2$} and {\color{red}$p=\infty$}.}
		\label{fig:Phi1}
\end{figure}

Brevig and Saksman \cite{BS20} have recently studied the extremal problem
\[\Psi_k(p) = \sup \left\{\mre\frac{f^{(k)}(0)}{k!}\,:\, \|f\|_{H^p}\leq 1 \right\}\]
for $0<p<1$. It is observed in \cite[Sec.~5.3]{BS20} that $\Psi_k(p) = \max_{0\leq t \leq 1} \Phi_k(p,t)$. In particular, the maxima of $\Phi_1(p,t)$ for $0 \leq t \leq 1$ is
\[\Psi_1(p)=(1-p/2)^{1/p}\frac{2}{\sqrt{p(2-p)}}\]
and this is attained for $t=(1-p/2)^{1/p}$. From the main result in \cite{BK04}, it is easy to see that $t \mapsto \Phi_1(p,t)$ is a decreasing function from $\Phi_1(p,0)=1$ to $\Phi_1(p,1)=0$ when $1 \leq p \leq \infty$. Similarly, our main result shows that $\Phi_1(p,t)$ is increasing from $\Phi_1(p,0)=1$ to the maxima mentioned above, then decreasing to $\Phi_1(p,1)=0$. Figure~\ref{fig:Phi1} contains the plot of $t \mapsto \Phi_1(p,t)$ for several values $0<p\leq\infty$, which illustrates this difference between $0<p<1$ and $1 \leq p \leq \infty$.

Another difference between $0<p<1$ and $1 \leq p \leq \infty$ appears when we consider $k\geq2$. Recall that in the latter case, we have $\Phi_k(p,t)=\Phi_1(p,t)$ for every $k\geq2$ and every $0\leq t \leq1$. In the former case, we only get from Theorem~\ref{thm:fwiener} that
\begin{equation} \label{eq:Phi1k}
	\Phi_1(p,t) \leq \Phi_k(p,t) \leq k^{1/p-1} \Phi_1(p,t).
\end{equation}
Theorem~\ref{thm:fwiener} also shows that the upper bound in \eqref{eq:Phi1k} is attained if and only if $t=1$, since trivially $\Phi_1(p,1)=0$ for every $0<p\leq\infty$. However, by adapting an example due to Hardy and Littlewood \cite{HL32}, it is easy to see that if $0<p<1$ and $0\leq t<1$ are fixed, then the exponent ${1/p-1}$ in \eqref{eq:Phi1k} cannot be improved as $k\to\infty$. In the final section of the paper, we present some evidence that the lower bound in \eqref{eq:Phi1k} can be attained for sufficiently large $t$, if $k\geq2$ and $0<p<1$ are fixed.

\subsection*{Organization} The present paper is organized into five additional sections and one appendix. In Section~\ref{sec:prelim}, we collect some preliminary results pertaining to $H^p$ and the structure of extremals for $\Phi_k(p,t)$. Section~\ref{sec:fwiener} is devoted to F.~Wiener's trick and the proof of Theorem~\ref{thm:fwiener}. A complete solution to the extremal problem $\Phi_1(p,t)$ for $0<p\leq\infty$ and $0\leq t \leq 1$ is presented in Section~\ref{sec:Phi0inf}. In Section~\ref{sec:Phi1inf}, we consider $\Phi_k(p,t)$ for $k\geq2$ and $1\leq p \leq \infty$ and study when the extremal is unique. Section~\ref{sec:Phi201} contains some remarks on $\Phi_k(p,t)$ for $k\geq2$ and $0<p<1$. Appendix~\ref{app:equation} contains the proof of a crucial lemma needed to resolve the extremal problem $\Phi_1(p,t)$ for $0<p<1$.

\subsection*{Acknowledgements} The authors extend their gratitude to Eero Saksman for a helpful discussion pertaining to Theorem~\ref{thm:fwiener}. They also thank the referee for a careful reading of the paper.

\section{Preliminaries} \label{sec:prelim}
Recall that for $0<p<\infty$, the Hardy space $H^p$ consists of the analytic functions $f$ in $\mathbb{D}$ for which the limit of integral means
\[\|f\|_{H^p}^p = \lim_{r\to 1^-} \int_0^{2\pi} |f(r e^{i\theta})|^p\,\frac{d\theta}{2\pi}\]
is finite. $H^\infty$ is the space of bounded analytic functions in $\mathbb{D}$, endowed with the norm $\|f\|_{H^\infty} = \sup_{|z|<1}|f(z)|$. It is well-known (see e.g.~\cite{Duren}) that $H^p$ is a Banach space when $1\leq p \leq \infty$ and a quasi-Banach space when $0<p<1$.

In the Banach space range $1 \leq p \leq \infty$, the triangle equality is
\begin{equation} \label{eq:triangle1}
	\|f+g\|_{H^p} \leq \|f\|_{H^p} + \|g\|_{H^p}.
\end{equation}
The Hardy space $H^p$ is strictly convex when $1<p<\infty$, which means that it is impossible to attain equality in \eqref{eq:triangle1} unless $g\equiv0$ or $f=\lambda g$ for a non-negative constant $\lambda$. $H^p$ is not strictly convex for $p=1$ and $p=\infty$, so in this case there are other ways to attain equality in \eqref{eq:triangle1}. In the range $0<p<1$, the triangle inequality takes the form
\begin{equation} \label{eq:triangle2}
	\|f+g\|_{H^p}^p \leq \|f\|_{H^p}^p + \|g\|_{H^p}^p,
\end{equation}
so here $H^p$ is not even locally convex \cite{DRS69}. Our first goal is to establish that the triangle inequality \eqref{eq:triangle2} is not attained unless $f\equiv0$ or $g\equiv0$. This result is probably known to experts, but we have not found it in the literature. 

If $f \in H^p$ for some $0<p\leq\infty$, then the boundary limit function
\begin{equation} \label{eq:blf}
	f^\ast(e^{i\theta}) = \lim_{r\to 1^-} f(re^{i\theta})
\end{equation}
exists for almost every $\theta$. Moreover, $f^\ast\in L^p= L^p([0,2\pi])$ and
\[\|f\|_{H^p}=\|f^\ast\|_{L^p}=\left(\int_0^{2\pi}\left|f^\ast(e^{i\theta})\right|^p\,\frac{d\theta}{2\pi}\right)^\frac{1}{p}\]
if $0<p<\infty$ and $\|f\|_{H^\infty}=\operatorname{ess\,sup}_{\theta} |f^\ast(e^{i\theta})|$. For simplicity, we henceforth omit the asterisk and write $f^\ast=f$ with the limit \eqref{eq:blf} in mind.

\begin{lemma} \label{lem:Delta01}
	Fix $0<p<1$ and suppose that $f,g \in H^p$. If
	\[\|f+g\|_{H^p}^p = \|f\|_{H^p}^p + \|g\|_{H^p}^p\]
	then either $f \equiv 0$ or $g \equiv 0$.
\end{lemma}

\begin{proof}
	We begin by looking at equality in the triangle inequality for $L^p$ in the range $0<p<1$. Here we have
	\begin{align*}
		\|f+g\|_{L^p}^p &= \int_0^{2\pi} \left|f(e^{i\theta})+g(e^{i\theta})\right|^p\,\frac{d\theta}{2\pi} \\
		&\leq \int_0^{2\pi} |f(e^{i\theta})|^p+|g(e^{i\theta})|^p\,\frac{d\theta}{2\pi} = \|f\|_{L^p}^p + \|g\|_{L^p}^p.
	\end{align*}
	We used the elementary estimate $|z+w|^p\leq|z|^p+|w|^p$ for complex numbers $z,w$ and $0<p<1$. It is easily verified that this estimate is attained if and only if $zw=0$. Consequently, 
	\[\|f+g\|_{L^p}^p = \|f\|_{L^p}^p+\|g\|_{L^p}^p\]
	if and only if $f(e^{i\theta})g(e^{i\theta})=0$ for almost every $\theta$. It is well-known (see \cite[Thm.~2.2]{Duren}) that the only function $h \in H^p$ whose boundary limit function \eqref{eq:blf} vanishes on a set of positive measure is $h \equiv 0$. Hence we conclude that either $f \equiv 0$ or $g \equiv 0$.
\end{proof}

Let us next establish a standard result on the structure of the extremals for the extremal problem \eqref{eq:Phi}. The first step is the following basic result.
\begin{lemma} \label{lem:norm1}
	If $f\in H^p$ is extremal for $\Phi_k(p,t)$, then $\|f\|_{H^p}=1$.
\end{lemma}

\begin{proof}
	Suppose that $f\in H^p$ is extremal for $\Phi_k(p,t)$ but that $\|f\|_{H^p}<1$. For $\varepsilon>0$, set $g(z) = f(z)+\varepsilon z^k$. Note that $g(0)=f(0)=t$ for any $\varepsilon>0$. If $1 \leq p \leq \infty$, then
	\begin{align*}
		\|g\|_{H^p} &\leq \|f\|_{H^p}+\varepsilon < 1
		\intertext{for sufficiently small $\varepsilon>0$. If $0<p<1$, then}
		\|g\|_{H^p}^p &\leq \|f\|_{H^p}^p + \varepsilon^p <1,
	\end{align*}
	again for sufficiently small $\varepsilon>0$, so $\|g\|_{H^p}<1$. In both cases we find that
	\[\frac{g^{(k)}(0)}{k!} = \frac{f^{(k)}(0)}{k!}+\varepsilon,\]
	which contradicts the extremality of $f$ for $\Phi_k(p,t)$.
\end{proof}

Let $(n_j)_{j=1}^k$ denote a sequence of distinct non-negative integers and let $(w_j)_{j=1}^k$ denote a sequence of complex numbers. A special case of the Carath\'eodory--Fej\'er problem is to determine the infimum of $\|f\|_{H^p}$ over all $f \in H^p$ which satisfy
\begin{equation} \label{eq:CF}
	\frac{f^{(n_j)}}{n_j!}(0)=w_j,
\end{equation}
for $j=1,\ldots, k$. Set $k=\max_{1 \leq j \leq k} n_j$. If $f$ is an extremal for the Carath\'eodory--Fej\'er problem \eqref{eq:CF}, then there are complex numbers $|\lambda_j|\leq1$ for $j=1,\ldots,k$ and a constant $C$ such that
\begin{equation} \label{eq:CFstructure}
	f(z) = C \prod_{j=1}^l \frac{\lambda_j-z}{1-\overline{\lambda_j} z} \prod_{j=1}^k (1-\overline{\lambda_j} z)^{2/p}
\end{equation}
for some $0 \leq l \leq k$, and the strict inequality $|\lambda_j|<1$ holds for $0< j \leq l$. In \eqref{eq:CFstructure} and in similar formulas to follow, we adopt the convention that in the case $l=0$ the first product is empty and considered to be equal to $1$.

For $1 \leq p \leq \infty$, this result is independently due to Macintyre--Rogosinski \cite{MR50} and Havinson \cite{Havinson51}, while in the range $0<p<1$ the result is due to Kabaila \cite{Kabaila60}. An exposition of these results can be found in \cite[Ch.~8]{Duren} and \cite[pp.~82--85]{Khavinson86}, respectively.

Using Lemma~\ref{lem:norm1}, we can establish that the extremals of the extremal problem $\Phi_k(p,t)$ have to be of the same form.

\begin{lemma} \label{lem:structure}
	If $f\in H^p$ is extremal for $\Phi_k(p,t)$, then there are complex numbers $|\lambda_j|\leq1$ for $j=1,\ldots,k$ and a constant $C$ such that
	\[f(z) = C \prod_{j=1}^l \frac{\lambda_j-z}{1-\overline{\lambda_j} z} \prod_{j=1}^k (1-\overline{\lambda_j} z)^{2/p}.\]
	for some $0 \leq l \leq k$, and the strict inequality $|\lambda_j|<1$ holds for $0< j \leq l$.
\end{lemma}

\begin{proof}
	Suppose that $f$ is extremal for $\Phi_k(p,t)$ and consider the Carath\'eodory--Fej\'er problem with conditions
	\begin{equation} \label{eq:CFcond}
		f(0)=t \qquad \text{and} \qquad \frac{f^{(k)}(0)}{k!} = \Phi_k(p,t) .
	\end{equation}
	We claim that $f$ is an extremal for the Carath\'eodory--Fej\'er problem \eqref{eq:CFcond}. If it is not, then there must be some $f \in H^p$ with $\|f\|_{H^p}<1$ which satisfies \eqref{eq:CFcond}. However, this contradicts Lemma~\ref{lem:norm1}. Hence the extremal is of the stated form by \eqref{eq:CFstructure}.
\end{proof}

\section{F.~Wiener's trick} \label{sec:fwiener}
Recall from \eqref{eq:fwienerdef} that if $f(z)=\sum_{n\geq0} a_n z^n$ and $\omega_k = \exp(2\pi i /k)$, then
\[W_k f(z) = \frac{1}{k}\sum_{j=0}^{k-1} f(\omega_k^j z) = \sum_{n=0}^\infty a_{kn}z^{kn}.\]
We begin by giving two examples showing that $\|W_k f \|_{H^p}=\|f\|_{H^p}$ may occur for $f$ such that $W_k f \neq f$ when $p=1$ or $p=\infty$.

\begin{example}
	Let $k\geq2$ and consider $f(z)=(1+z)^{2k}$ in $H^1$. By the binomial theorem, we find that
	\begin{align*}
		f(z) &= \sum_{n=0}^{2k} \binom{2k}{n} z^n, \\
		W_kf(z) &= 1 + \binom{2k}{k}z^k + z^{2k}.
	\end{align*}
	Note that $f \neq W_k f$ since $k\geq2$. By another application of the binomial theorem and a well-known identity for the central binomial coefficient, we find that
	\[\|f\|_{H^1} = \|f^{1/2}\|_{H^2}^{2} = \sum_{n=0}^k \binom{k}{n}^2 = \binom{2k}{k}.\]
	Moreover,
	\[\binom{2k}{k} = \int_0^{2\pi} W_k f(e^{i\theta}) \,\overline{e^{ik\theta}}\,\frac{d\theta}{2\pi} \leq \|W_k f\|_{H^1}\]
	by the triangle inequality. Hence
	\[\binom{2k}{k} \leq \|W_k f\|_{H^1} \leq \|f\|_{H^1} = \binom{2k}{k},\]
	so $\|W_k f\|_{H^1} = \|f\|_{H^1}$.
\end{example}

\begin{example}
	Let $k\geq2$ and consider $f(z)=(1+z^k)^2 - z(1-z^k)^2$ in $H^\infty$. It is clear that $W_k f (z) = (1+z^k)^2 \neq f(z)$ since $k\geq2$. Moreover $\|W_k f\|_{H^\infty}=4$. The supremum is attained for $z=\omega_k^j$ for $j=0,1,\ldots,k-1$. We next compute
	\[f(e^{i\theta}) = (1+e^{ik\theta})^2 - e^{i\theta}(1-e^{ik\theta})^2 = 4 e^{ik\theta}\left(\cos^2(k\theta/2)+e^{i\theta}\sin^2(k\theta/2)\right).\]
	Consequently, $\|f\|_{H^\infty}=4$ and here the supremum is attained for $z=\omega_{2k}^{j}$ for $j=0,1,\ldots,2k-1$.
\end{example}

\begin{proof}[Proof of Theorem~\ref{thm:fwiener}]
	It follows from the triangle inequality \eqref{eq:triangle1} that 
	\begin{equation} \label{eq:fwiener1inf}
		\|W_k f\|_{H^p} \leq \|f\|_{H^p}
	\end{equation}
	for every $f \in H^p$ if $1 \leq p \leq \infty$. In the range $0<p<1$, we get from the triangle inequality \eqref{eq:triangle2} the estimate
	\begin{equation} \label{eq:fwiener01}
		\|W_k f\|_{H^p} \leq k^{1/p-1} \|f\|_{H^p}
	\end{equation}
	for every $f \in H^p$. Combining \eqref{eq:fwiener1inf} and \eqref{eq:fwiener01}, we have established that
	\[\|W_k f\|_{H^p} \leq \max\big(k^{1/p-1},1\big)\|f\|_{H^p}.\]
	This is trivially attained for $f(z)=z^k$ when $1 \leq p \leq \infty$. We need to show that the upper bound $k^{1/p-1}$ cannot be improved when $0<p<1$ to finish proof of the first part of the theorem. 
	
	Let $\varepsilon>0$ and consider $f_\varepsilon(z) = \left(z-(1+\varepsilon)\right)^{-1/p}$. Clearly $\|f_\varepsilon\|_{H^p} \to \infty$ as $\varepsilon\to 0^+$. Moreover
	\begin{align*}
		\|f_\varepsilon\|_{H^p}^p &= \int_0^{2\pi} \frac{1}{|e^{i\theta}-(1+\varepsilon)|}\,\frac{d\theta}{2\pi}\\
		&\leq \int_{|\theta| < \pi/k} \frac{1}{|e^{i\theta}-(1+\varepsilon)|}\,\frac{d\theta}{2\pi} + \int_{|\theta|\geq \pi/k} \frac{6}{\theta^2} \,\frac{d\theta}{2\pi} \\
		&\leq \int_{|\theta| < \pi/k} \frac{1}{|e^{i\theta}-(1+\varepsilon)|}\,\frac{d\theta}{2\pi} + \frac{6k}{\pi^2},
	\end{align*}
	from which we conclude that
	\begin{equation} \label{eq:mih}
		 \|f_\varepsilon\|_{H^p}^p = \int_{|\theta| < \pi/k} \frac{1}{|e^{i\theta}-(1+\varepsilon)|}\,\frac{d\theta}{2\pi} + O(1).
	\end{equation}
	Furthermore,
	\begin{align*}
		\|W_k f_\varepsilon\|_{H^p}^p &= \sum_{j=0}^{k-1} \int_{|\theta-2\pi j/k|<\pi/k} \left|\sum_{l=0}^{k-1} \frac{f_\varepsilon\big(e^{i(\theta+2\pi l/k)}\big)}{k}\right|^p \,\frac{d\theta}{2\pi} \\
		&\geq k^{-p} \sum_{j=0}^{k-1} \left(\int_{|\theta-2\pi j/k|<\pi/k} \big|f_\varepsilon\big(e^{i(\theta+2\pi j/k)}\big)\big|^p\,\frac{d\theta}{2\pi} - \frac{6k^2}{\pi^2}\right) \\
		&= k^{-p+1} \int_{|\theta| < \pi/k} \frac{1}{|e^{i\theta}-(1+\varepsilon)|}\,\frac{d\theta}{2\pi} - \frac{6k^{-p+3}}{\pi^2}.
	\end{align*}
	By \eqref{eq:mih} we find that
	\[\lim_{\varepsilon\to0^{+}}\frac{\|W_k f_\varepsilon\|_{H^p}^p}{\|f_\varepsilon\|_{H^p}^p} \geq k^{1-p}.\]
	Hence, the constant $k^{1/p-1}$ in \eqref{eq:fwiener01} cannot be replaced by any smaller quantity.	
	
	We next want to show that (a) and (b) holds. For a function $f \in H^p$, define $f_j(z)=f(\omega_k^j z)$ for $j=0,1,\ldots,k-1$ and recall that $\|f\|_{H^p} = \|f_j\|_{H^p}$.
	
	We begin with (a). Suppose that $\|W_k f\|_{H^p}=k^{1/p-1}\|f\|_{H^p}$, which we can reformulate as
	\[\|f_0 + f_1 + \cdots + f_{k-1}\|_{H^p}^p = \|f_0\|_{H^p}^p+\|f_1\|_{H^p}^p+\cdots+\|f_{k-1}\|_{H^p}^p.\]
	By Lemma~\ref{lem:Delta01}, the triangle inequality can be attained if and only if at least $k-1$ of the $k$ functions $f_j$ are identically equal to zero. Evidently this is possible if and only if $f\equiv0$.
	
	For (b), we suppose that $f\in H^p$ is such that $\|W_k f\|_{H^p}=\|f\|_{H^p}$. We need to prove that $W_k f = f$. If $f \equiv 0$ there is nothing to do. As in the proof of (a), we note that $\|W_k f\|_{H^p}=\|f\|_{H^p}$ can be reformulated as
	\[\|f_0 + f_1 + \cdots + f_{k-1}\|_{H^p} = \|f_0\|_{H^p}+\|f_1\|_{H^p}+\cdots+\|f_{k-1}\|_{H^p}.\]
	Viewing $H^p$ as a subspace of $L^p$, the strict convexity of the latter implies that there are non-negative constants $\lambda_j$ for $j=1,2,\ldots,k-1$ such that
	\[f = f_0 = \lambda_1 f_1 = \cdots = \lambda_{k-1} f_{k-1}.\]
	We shall only look at $f = \lambda_1 f_1$ which for $f(z)=\sum_{n\geq0} a_n z^n$ is equivalent to
	\[\sum_{n=0}^\infty a_n z^n = \lambda_1 \sum_{n=0}^\infty a_n \omega_k^n z^n.\]
	Using $W_k$ on this identity we get
	\[\sum_{n=0}^\infty a_{kn} z^{kn} = \lambda_1 \sum_{n=0}^\infty a_{kn} z^{kn}.\]
	This is only possible if $\lambda_1=1$ or $W_k f \equiv 0$. The latter implies that $f \equiv 0$ since $\|W_k f \|_{H^p}=\|f\|_{H^p}$ by assumption. Therefore we can restrict our attention to the case $\lambda_1=1$. For all integers $n$ that are not a multiple of $k$, we now find that
	\[a_n = \lambda_1 \omega_k^n a_n \qquad \implies \qquad  a_n =0,\]
	since $\lambda_1=1$ and $\omega_k^n \neq 1$. Hence $W_k f = f$ as desired.
\end{proof}

Recall that a function $f \in H^p$ is called inner if $|f(e^{i\theta})|=1$ for almost every $\theta$. We shall require the following simple result later on.

\begin{lemma} \label{lem:inner}
	If both $f$ and $W_k f$ are inner functions, then $f = W_k f$.
\end{lemma}

\begin{proof}
    Since $|W_k f(e^{i\theta})|=|f(e^{i\theta})|=1$ for almost every $\theta$, we get from \eqref{eq:fwienerdef} that
   	\begin{equation} \label{eq:WkB}
   		1 = |W_kf(e^{i\theta})| = \Bigg|\frac{1}{k}\sum_{j=0}^{k-1} f_j(e^{i\theta})\Bigg| = \frac{1}{k}\sum_{j=0}^{k-1} |f_j(e^{i\theta})|,
   	\end{equation}
   	where $f_j(z)=f(\omega_k^j z)$. The equality on the right hand side of \eqref{eq:WkB} is possible if and only if
   	\[f(e^{i\theta})=f_1(e^{i\theta})=\cdots=f_{k-1}(e^{i\theta})\]
   	for almost every $\theta$. As in the proof of Theorem~\ref{thm:fwiener} (b), we find that $f = W_k f$.
\end{proof}

\section{The extremal problem $\Phi_1(p,t)$ for $0 < p \leq \infty$} \label{sec:Phi0inf}
In the present section, we resolve the extremal problem \eqref{eq:Phi} in the case $k=1$ completely. We begin with the case $1 \leq p \leq \infty$ which has been solved by Beneteau and Korenblum \cite{BK04}. We give a different proof of their result based on Lemma~\ref{lem:structure}, mainly to illustrate the differences between the cases $0<p<1$ and $1 \leq p \leq \infty$.

\begin{theorem}[Beneteau--Korenblum] \label{thm:BK1}
	Fix $1 \leq p \leq \infty$ and consider \eqref{eq:Phi} with $k=1$. 
	\begin{enumerate}
		\item[(i)] If $0 \leq t < 2^{-1/p}$, let $\alpha$ denote the unique real number in the interval $0 \leq \alpha < 1$ such that $t = \alpha(1+\alpha^2)^{-1/p}$. Then
			\begin{align*}
				\Phi_1(p,t) &= \frac{1}{\left(1+\alpha^2\right)^{1/p}}\left(1+ \left(\frac{2}{p}-1\right)\alpha^2\right),\phantom{\frac{2\beta}{p}}
				\intertext{and the unique extremal is}
				f(z) &= \frac{\alpha+z}{1+\alpha z} \frac{\left(1+\alpha z\right)^{2/p}}{\left(1+\alpha^2\right)^{1/p}}.
			\end{align*}
		\item[(ii)] If $2^{-1/p} \leq t \leq 1$, let $\beta$ denote the unique real number in the interval $0 \leq \beta \leq 1$ such that $t = (1+\beta^2)^{-1/p}$. Then
			\begin{align*}
				\Phi_1(p,t) &= \frac{1}{\left(1+\beta^2\right)^{1/p}} \frac{2\beta}{p},\phantom{\left(1+ \left(\frac{2}{p}-1\right)\alpha^2\right)}
				\intertext{and the unique extremal is}
				f(z) &= \frac{\left(1+\beta z\right)^{2/p}}{\left(1+\beta^2\right)^{1/p}}.
			\end{align*}
	\end{enumerate}
\end{theorem}

\begin{proof}
	Note that since $k=1$, there are only two possibilities for the extremals in Lemma~\ref{lem:structure}. They are
	\begin{align}
		f_1(z) &= \frac{\alpha+z}{1+\alpha z} \frac{\left(1+\alpha z\right)^{2/p}}{\left(1+\alpha^2\right)^{1/p}}, & & 0\leq \alpha < 1, \label{eq:f1} \\
		f_2(z) &= \frac{\left(1+\beta z\right)^{2/p}}{\left(1+\beta^2\right)^{1/p}}, && 0 \leq \beta \leq 1. \label{eq:f2}
	\end{align}
	Here we have made $\alpha,\beta\geq0$ by rotations. Note that if $p=\infty$, then $f_2$ does not depend on $\beta$. Moreover,
	\begin{align}
		t&= f_1(0) = \frac{\alpha}{(1+\alpha^2)^{1/p}}, \label{eq:talpha} \\
		t&= f_2(0) = \frac{1}{(1+\beta^2)^{1/p}}. \label{eq:tbeta}
	\end{align}
	For $1 \leq p \leq \infty$ it is easy to verify that the function
	\begin{equation} \label{eq:alphafunc}
		\alpha \mapsto\frac{\alpha}{(1+\alpha^2)^{1/p}}
	\end{equation}
	is strictly increasing on $0 \leq \alpha <1$ and maps $[0,1)$ to $[0,2^{-1/p})$. Similarly, for $1 \leq p < \infty$ we find that the function 
	\begin{equation} \label{eq:betafunc}
		\beta \mapsto \frac{1}{(1+\beta^2)^{1/p}}
	\end{equation}
	is strictly decreasing on $0 \leq \beta \leq 1$ and maps $[0,1]$ to $[2^{-1/p},1]$. Consequently, if $0 \leq t < 2^{-1/p}$, then the unique extremal is \eqref{eq:f1} with $\alpha$ given by \eqref{eq:talpha}, and if $2^{-1/p} \leq t \leq 1$, then the unique extremal is \eqref{eq:f2} with $\beta$ given by \eqref{eq:tbeta}. The proof is completed by computing 
	\begin{align}
		f_1'(0) &= \frac{1}{\left(1+\alpha^2\right)^{1/p}}\left(1+\alpha^2\left(\frac{2}{p}-1\right)\right) = t \left(\frac{1}{\alpha}+\alpha\left(\frac{2}{p}-1\right)\right), \label{eq:f10} \\
		f_2'(0) &= \frac{1}{(1+\beta^2)^{1/p}} \frac{2\beta}{p} = t \frac{2\beta}{p},\label{eq:f20}
	\end{align}
	to obtain the stated expressions for $\Phi_1(p,t)$ in (i) and (ii), respectively.
\end{proof}

Define $\alpha$ and $\beta$ as functions of $t$ implicitly through \eqref{eq:talpha} and \eqref{eq:tbeta}. Then $\alpha$ is increasing on $0 \leq t < 2^{-1/p}$ and $\beta$ is decreasing on $2^{-1/p}\leq t \leq 1$. Inspecting the left hand side of \eqref{eq:f10} and \eqref{eq:f20}, we extract the following result.
\begin{corollary} \label{cor:dec}
	If $1 \leq p \leq \infty$, then the function $t \mapsto \Phi_1(p,t)$ is decreasing and takes the values $[0,1]$.
\end{corollary}

In the range $0<p<1$ a more careful analysis is required. This is due to the fact that the function \eqref{eq:alphafunc} is increasing on the interval $0 \leq \alpha \leq \alpha_2$ and decreasing on the interval $\alpha_2 \leq \alpha < 1$, where
\begin{equation} \label{eq:alpha2}
	\alpha_2 = \sqrt{\frac{p}{2-p}}.
\end{equation}
Inspecting \eqref{eq:talpha}, we conclude that for each $2^{-1/p}<t<2^{-1/p}\sqrt{p}(2-p)^{1/p-1/2}$ there are two possible $\alpha$-values which give the same $t=f_1(0)$. Let $\alpha_1$ denote the unique real number in the interval $(0,1)$ such that
\begin{equation} \label{eq:alpha1}
	1+\alpha_1^2= 2\alpha_1^p.
\end{equation}
Note that $\alpha_1$ gives the value $t=2^{-1/p}$ in \eqref{eq:talpha}. 

\begin{lemma} \label{lem:smallisbig}
	If $\alpha_1 < \alpha< \alpha_2$ and $\alpha_2<\widetilde{\alpha}<1$ produce the same $t=f_1(0)$ in \eqref{eq:talpha}, then the quantity $f'_1(0)$ from \eqref{eq:f10} is maximized by $\alpha$.
\end{lemma}

\begin{proof}
	Since $\alpha$ and $\widetilde{\alpha}$ give the same $t=f_1(0)$ in \eqref{eq:f10}, we only need to prove that 
	\begin{equation} \label{eq:alphatilde}
		\frac{1}{\alpha}+\frac{\alpha}{\alpha_2^2}>\frac{1}{\widetilde{\alpha}}+\frac{\widetilde{\alpha}}{\alpha_2^2}.
	\end{equation}
	Fix $\alpha_1 < \alpha <\alpha_2$. The unique number $\alpha_2<\xi<1$ such that 
	\[\frac{1}{\alpha}+\frac{\alpha}{\alpha_2^2} = \frac{1}{\xi}+\frac{\xi}{\alpha_2^2}\]
	is $\xi=\alpha_2^2/\alpha$. Since the function
	\[x \mapsto \frac{1}{x}+\frac{x}{\alpha_2^2}\]
	is increasing for $x>\alpha_2$ it is sufficient to prove that $\xi>\widetilde{\alpha}$ to obtain \eqref{eq:alphatilde}. Since
	\[x \mapsto \frac{x}{(1+x^2)^{1/p}}\]
	is decreasing for $x>\alpha_2$, we see that $\xi>\widetilde{\alpha}$ if and only if
	\[\frac{\widetilde{\alpha}}{(1+\widetilde{\alpha}^2)^{1/p}}>\frac{\xi}{(1+\xi^2)^{1/p}} \qquad \Longleftrightarrow \qquad \frac{\alpha}{(1+\alpha^2)^{1/p}}> \frac{\alpha_2^2/\alpha}{(1+(\alpha_2^2/\alpha)^2)^{1/p}}.\]
	Here we used that $\alpha$ and $\widetilde{\alpha}$ give the same $t=f_1(0)$ in \eqref{eq:talpha} on the left hand side and the identity $\xi = \alpha_2^2/\alpha$ on the right hand side. We now substitute $\alpha = \alpha_2 \sqrt{x}$ for $0<x<1$ to obtain the equivalent inequality 
	\begin{equation} \label{eq:xalpha}
		\frac{x}{(1+\alpha_2^2 x)^{1/p}}>\frac{1}{(1+\alpha_2^2/x)^{1/p}}.
	\end{equation}
	Actually, we only need to consider $(\alpha_1/\alpha_2)^2<x<1$, but the same proof works for $0<x<1$. We raise both sides of \eqref{eq:xalpha} to the power $p$, multiply by $x^{1-p}$ and rearrange to get the equivalent inequality $F(x)>0$ where
	\[F(x) = \left(x-x^{1-p}\right)+\alpha_2^2 \left(1-x^{2-p}\right).\]
	Recalling that $\alpha_2^2 = p/(2-p)$, we compute
	\[F'(x) = \left(1-(1-p)x^{-p}\right) - px^{1-p} \quad \text{and}\quad F''(x) = p(1-p)x^{-p-1}-p(1-p)x^{-p}.\]
	Since $F(1)=F'(1)=0$, we get from Taylor's theorem that for every $0<x<1$ there is some $x<\eta<1$ such that
	\[F(x) = \frac{F''(\eta)}{2}(x-1)^2 = \frac{p(1-p)}{2}\eta^{-p}\left(\eta^{-1}-1\right)(x-1)^2 >0,\]
	which completes the proof.
\end{proof}

By Lemma~\ref{lem:smallisbig}, we now only need to compare $f_1'(0)$ from \eqref{eq:f10} for $\alpha_1 \leq \alpha \leq \alpha_2$ with $f_2'(0)$ from \eqref{eq:f20} for $\beta$ such that $f_1(0)=t=f_2(0)$. Inspecting \eqref{eq:talpha} and \eqref{eq:tbeta}, we find that  
\begin{equation} \label{eq:beta}
	\frac{\alpha}{\left(1+\alpha^2\right)^{1/p}} = \frac{1}{\left(1+\beta^2\right)^{1/p}} \qquad \Longleftrightarrow \qquad \beta = \sqrt{\frac{1+\alpha^2}{\alpha^p}-1}.
\end{equation}
Next, we consider the equation $f_1'(0)=f_2'(0)$ with $\beta$ as in \eqref{eq:beta}. Inspecting \eqref{eq:f10} and \eqref{eq:f20} and dividing by $t$, we get the equation
\begin{equation} \label{eq:alpha}
	\frac{1}{\alpha}+\alpha\left(\frac{2}{p}-1\right) = \frac{2\beta}{p} = \frac{2}{p}\sqrt{\frac{1+\alpha^2}{\alpha^p}-1}.
\end{equation}
We square both sides, multiply by $p^2$ and rearrange to find that \eqref{eq:alpha} is equivalent to the equation $F_p(\alpha)=0$, where
\begin{equation} \label{eq:Fpa}
	F_p(\alpha) = p^2 \alpha^{-2}+2p(2-p) + (2-p)^2 \alpha^2 - 4\left(\alpha^{-p}+\alpha^{2-p}-1\right).
\end{equation}
Suppose that $\alpha_1 \leq \alpha \leq \alpha_2$. If
\begin{itemize}
	\item $F_p(\alpha)>0$, then $f_1$ from \eqref{eq:f1} is the unique extremal for $\Phi_1(p,t)$.
	\item $F_p(\alpha)=0$, then $f_1$ from \eqref{eq:f1} and $f_2$ from \eqref{eq:f2} are extremals for $\Phi_1(p,t)$.
	\item $F_p(\alpha)<0$, then $f_2$ from \eqref{eq:f2} is the unique extremal for $\Phi_1(p,t)$.
\end{itemize}
Note that any solutions of $F_p(\alpha)=0$ with $0<\alpha<\alpha_1$ are of no interest since this implies that $\beta>1$ by \eqref{eq:beta}. Similarly, any solutions of $F_p(\alpha)=0$ with $\alpha_2<\alpha<1$ can be ignored due to Lemma~\ref{lem:smallisbig}. The following result shows that there is only one solution, which is in the pertinent range.

\begin{lemma} \label{lem:equation}
	Let $F_p$ be as in \eqref{eq:Fpa}. The equation $F_p(\alpha)=0$ has a unique solution, denoted $\alpha_p$, on the interval $(0,1)$. Moreover,
	\begin{enumerate}
		\item[(a)] if $0<\alpha<\alpha_p$, then $F_p(\alpha)>0$.
		\item[(b)] if $\alpha_p<\alpha<1$, then $F_p(\alpha)<0$.
		\item[(c)] $\alpha_1<\alpha_p<\alpha_2$ where $\alpha_1$ and $\alpha_2$ are from \eqref{eq:alpha1} and \eqref{eq:alpha2}, respectively.
	\end{enumerate}
\end{lemma}

The proof of Lemma~\ref{lem:equation} is a rather laborious calculus exercise, which we postpone to Appendix~\ref{app:equation} below. Let $\alpha_p$ be as in Lemma~\ref{lem:equation} and define
\begin{equation} \label{eq:tp}
	t_p = \frac{\alpha_p}{\left(1+\alpha_p^2\right)^{1/p}}.
\end{equation}
Note that $2^{-1/p}< t_p < 2^{-1/p} \sqrt{p}(2-p)^{1/p-1/2}$ by the fact that $\alpha_1 < \alpha_p < \alpha_2$. By the analysis above, Lemma~\ref{lem:smallisbig} and Lemma~\ref{lem:equation}, we obtain the following version of Theorem~\ref{thm:BK1} in the range $0<p<1$.

\begin{theorem} \label{thm:Phi01}
	Fix $0<p<1$ and consider \eqref{eq:Phi} with $k=1$. Let $t_p$ be as in \eqref{eq:tp} and set $\alpha_2 = \sqrt{p/(2-p)}$.
	\begin{enumerate}
		\item[(i)] If $0 \leq t \leq t_p$, let $\alpha$ denote the unique real number in the interval $0 \leq \alpha < \alpha_2$ such that $t = \alpha(1+\alpha^2)^{-1/p}$. Then
			\begin{align*}
				\Phi_1(p,t) &= \frac{1}{\left(1+\alpha^2\right)^{1/p}}\left(1+ \left(\frac{2}{p}-1\right)\alpha^2\right),\phantom{\frac{2\beta}{p}}
				\intertext{and an extremal is}
				f(z) &= \frac{\alpha+z}{1+\alpha z} \frac{\left(1+\alpha z\right)^{2/p}}{\left(1+\alpha^2\right)^{1/p}}.
			\end{align*}
		\item[(ii)] If $t_p \leq t \leq 1$, let $\beta$ denote the unique real number in the interval $0 \leq \beta \leq 1$ such that $t = (1+\beta^2)^{-1/p}$. Then
			\begin{align*}
				\Phi_1(p,t) &= \frac{1}{\left(1+\beta^2\right)^{1/p}} \frac{2\beta}{p},\phantom{\left(1+ \left(\frac{2}{p}-1\right)\alpha^2\right)}
				\intertext{and an extremal is}
				f(z) &= \frac{\left(1+\beta z\right)^{2/p}}{\left(1+\beta^2\right)^{1/p}}.
			\end{align*}
	\end{enumerate}
	The extremals are unique for $0 \leq t \neq t_p \leq 1$. The only extremals for $\Phi_1(p,t_p)$ are the functions given in \textnormal{(i)} and \textnormal{(ii)}.
\end{theorem}

\begin{figure}
	\centering
	\begin{tikzpicture}
		\begin{axis}
			[axis equal image,
			axis lines=middle,
			axis line style=thin,
			axis on top,
			xmin=0,
			xmax=1.1,
			xtick={0.25,0.5,0.75,1},
			xticklabels={\sfrac{1}{4},\sfrac{1}{2},\sfrac{3}{4},1},
			ymin=0,
			ymax=.55,
			ytick={.25,.5},
			yticklabels={\sfrac{1}{4},\sfrac{1}{2}},
			ylabel=$t$,
			xlabel=$p$,
			every axis x label/.style={ at={(ticklabel* cs:1.025)}, anchor=west,},
			every axis y label/.style={ at={(ticklabel* cs:1.025)}, anchor=south,},
			axis line style={->}, x post scale=1.5,
			axis line style={->}, y post scale=1.6855]
			\addplot[thin, name path=t1, densely dotted, color=blue, draw opacity=1, domain=0:1, samples=500] {2^(-1/x)};
			\addplot[thin, name path=t2, densely dotted, color=blue, draw opacity=1, domain=0:1, samples=500] {(1-x/2)^(1/x)*(x/(2-x))^(1/2)};
			\addplot[thin, blue, draw opacity=0, name path=l1, domain=0:1] {0};
			\addplot[thin, blue, draw opacity=0, name path=l2, domain=0:1] {0.5};
			
			\addplot [
			    thick,
			    color=blue,
			    fill=blue, 
			    fill opacity=0.05
			]
			fill between[
			    of=t1 and l1,
			    soft clip={domain=0:1},
			];
			
			\addplot [
			    thick,
			    color=blue,
			    fill=blue, 
			    fill opacity=0.05
			]
			fill between[
			    of=l2 and t2,
			    soft clip={domain=0:1},
			];
			\input{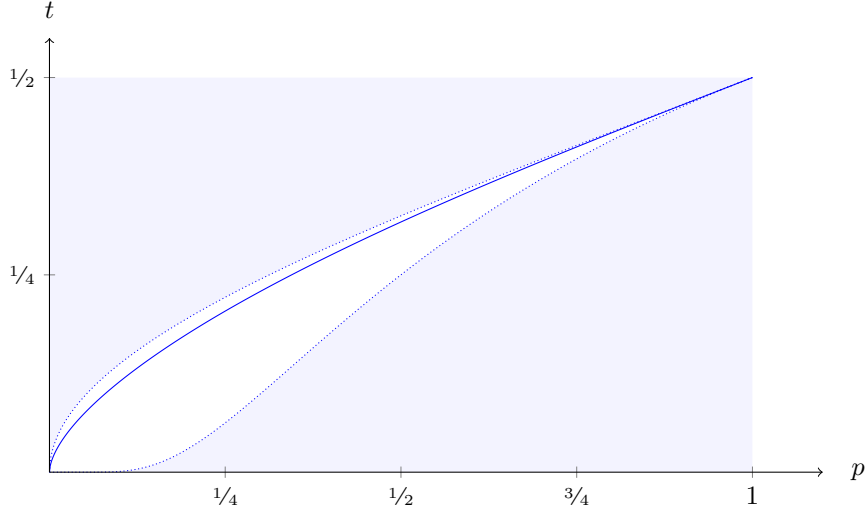}
		\end{axis}
	\end{tikzpicture}
	\caption{\label{fig:t_p} Plot of the curve $p \mapsto t_p$. Points $(p,t)$ above and below the curve correspond to the cases (i) and (ii) of Theorem~\ref{thm:Phi01}, respectively. The estimates $2^{-1/p}<t_p<2^{-1/p}\sqrt{p}(2-p)^{1/p-1/2}$ are represented by dotted curves. In the shaded area and in the range $1/2 \leq t \leq 1$, Theorem~\ref{thm:Phi01} is originally due to Connelly \cite{Con17}.}
\end{figure}

Theorem~\ref{thm:Phi01} extends \cite[Thm.~4.1]{Con17} to general $0 \leq t \leq1$. The analysis in \cite{Con17} is similar to ours, and we are able to also identify the extremals in the range
\[2^{-1/p}\leq t \leq 2^{-1/p}\sqrt{p}(2-p)^{1/p-1/2}\]
due to Lemma~\ref{lem:smallisbig} and Lemma~\ref{lem:equation}. It is also demonstrated in \cite[Thm.~4.1]{Con17} that when $p=1/2$ there must exist at least one value of $0<t<1$ for which the extremal is not unique. Theorem~\ref{thm:Phi01} shows that there is precisely one such $t$ and that this observation is not specific to $p=1/2$, but in fact holds for any $0<p<1$. Figure~\ref{fig:t_p} shows the value $t_p$ for which the extremal is not unique as a function of $p$.

Inspecting Theorem~\ref{thm:Phi01}, we get the following result similarly to how we extracted Corollary~\ref{cor:dec} from Theorem~\ref{thm:BK1}.

\begin{corollary} \label{cor:inc}
	If $0<p<1$, then the function $t \mapsto \Phi_1(p,t)$ is increasing from $\Phi_1(p,0)=1$ to
	\[\Phi_1\big(p,(1-p/2)^{1/p}\big)=(1-p/2)^{1/p}\frac{2}{\sqrt{p(2-p)}}\]
	and then decreasing to $\Phi_1(p,1)=0$.
\end{corollary}

\section{The extremal problem $\Phi_k(p,t)$ for $k\geq 2$ and $1 \leq p \leq \infty$} \label{sec:Phi1inf}
We begin by recalling how F.~Wiener's trick was used in \cite{BK04} to obtain the solution to the extremal problem $\Phi_k(p,t)$ for $k\geq2$ from Theorem~\ref{thm:BK1}.

\begin{theorem}[Benetau--Korenblum] \label{thm:BK2}
	Let $k\geq2$ be an integer. For every $1 \leq p \leq \infty$ and every $0 \leq t \leq 1$,
	\[\Phi_k(p,t)=\Phi_1(p,t).\]
	If $f_1$ is the extremal function for $\Phi_1(p,t)$, then $f_k(z)=f_1(z^k)$ is an extremal function for $\Phi_k(p,t)$.
\end{theorem}

\begin{proof}
	Suppose that $f$ is an extremal for $\Phi_k(p,t)$. Since $\|W_k f\|_{H^p} \leq \|f\|_{H^p}$, 
	\[f(0)=W_k f(0) \qquad \text{and} \qquad \frac{f^{(k)}(0)}{k!} = \frac{(W_kf)^{(k)}(0)}{k!},\]
	we conclude that $W_k f$ is also an extremal for $\Phi_k(p,t)$. Thus we may restrict our attention to extremals $\widetilde{f_k}$ of the form $\widetilde{f_k}(z)=\widetilde{f}(z^k)$ for $\widetilde{f} \in H^p$. The stated claims now follow at once from Theorem~\ref{thm:BK1}, since $\|\widetilde{f_k}\|_{H^p}=\|\widetilde{f}\|_{H^p}$.
\end{proof}

The purpose of the present section is to answer the following question. For which trios $k\geq2$, $1 \leq p \leq \infty$ and $0 \leq t \leq 1$ is the extremal for $\Phi_k(p,t)$ unique? Note that while Theorem~\ref{thm:BK2} provides an extremal $f_k(z)=f_1(z^k)$ where $f_1$ denotes the extremal from (the statement of) Theorem~\ref{thm:BK1}, it might not be unique.

In the case $1<p<\infty$ it follows at once from Theorem~\ref{thm:fwiener} (b) that this extremal is unique, although it is perhaps easier to use the strict convexity of $H^p$ and Lemma~\ref{lem:norm1} directly. Since $H^p$ is not strictly convex for $p=1$ and $p=\infty$, these cases require further analysis. Note that the case (a) below is certainly known to experts as a consequence of the general theory developed in \cite{Havinson51,MR50,RS53}.

\begin{theorem} \label{thm:unique}
	Consider the extremal problem \eqref{eq:Phi} for $k\geq2$ and $1 \leq p \leq \infty$.
	\begin{enumerate}
		\item[(a)] If $1<p\leq \infty$, then the unique extremal is $f_k(z)=f_1(z^k)$.
		\item[(b)] If $p=1$ and $1/2 \leq t \leq 1$, then the unique extremal is $f_k(z)=f_1(z^k)$.
		\item[(c)] If $p=1$ and $0 \leq t < 1/2$, then the extremals are the functions of the form
			\[f(z) = C \prod_{j=1}^k \big(\lambda_j-z\big)\big(1-\overline{\lambda_j}z\big)\]
			with $|\lambda_j|\leq1$ such that $\|f\|_{H^1}=1$, $f(0)=t$ and $f^{(k)}(0)>0$.
	\end{enumerate}
\end{theorem}

\begin{proof}[Proof of Theorem~\ref{thm:unique} {\normalfont(a)}]
	In view of the discussion above, we need only consider the case $p=\infty$. By Lemma~\ref{lem:structure}, we know that any extremal must be of the form
	\begin{equation} \label{eq:infext}
		f(z) = e^{i\theta} \prod_{j=1}^l \frac{\lambda_j-z}{1-\overline{\lambda_j} z}
	\end{equation}
	for some $0 \leq l \leq k$, constants $\lambda_j \in \mathbb{D}$ and $\theta \in \mathbb{R}$. If $f$ is extremal for $\Phi_k(\infty,t)$, then so is $W_k f$ by Theorem~\ref{thm:BK2}. Consequently, $W_k f$ is also of the form \eqref{eq:infext}. In particular, since both $f$ and $W_k f$ are inner, we get from Lemma~\ref{lem:inner} that $f = W_k f$. From the definition of $W_k$, we know that $f(z)=W_k f(z)=g(z^k)$ for some analytic function $g$. This shows that the only possibility in \eqref{eq:infext} is
	\[f(z) = e^{i\theta} \frac{\lambda-z^k}{1-\overline{\lambda}z^k}\]
	for some $\lambda \in \mathbb{D}$ and $\theta \in \mathbb{R}$. The unique extremal has $\theta=\pi$ and $\lambda=-t$.
\end{proof}

\begin{proof}[Proof of Theorem \ref{thm:unique} {\normalfont(b)}]
	Suppose that $f$ is extremal for $\Phi_k(1,t)$. By rotations, we extend our scope to functions $f$ such that $|f(0)|=t$. In this case, we can use Lemma~\ref{lem:structure} and write $f=gh$ for
	\begin{align*}
		g(z) &= C \prod_{j=1}^l (z+\alpha_j)\prod_{j=l+1}^k(1+\overline{\alpha_j}z), \\
		h(z) &= C \prod_{j=1}^k (1+\overline{\alpha_j}z).
	\end{align*}
    The constant $C>0$ satisfies 
	\[\frac{1}{C^2} = \sum_{j=0}^k \Bigg|\sum_{j_1+j_2+\cdots+j_k = j} \alpha_1^{j_1}\alpha_2^{j_2}\cdots \alpha_k^{j_k}\Bigg|^2 ,\]
	where $j_1,j_2,\ldots,j_k$ take only the values $0$ and $1$. Evidently $\|g\|_{H^2}=\|h\|_{H^2}=1$. Set $A_l = |\alpha_1 \cdots \alpha_l|$ and $B_l = |\alpha_{l+1}\cdots \alpha_k|$. By keeping only the terms $j=0$ and $j=k$ we obtain the trivial estimate
	\begin{equation} \label{eq:trivialest}
		\frac{1}{C^2} \geq 1 + |\alpha_1\alpha_2\cdots \alpha_k|^2 = 1 + A_l^2 B_l^2.
	\end{equation}
	We will adapt an argument due to F.~Riesz \cite{FRiesz19} to get some additional information on the relationship between $g$ and $h$. Write
	\[f(z) = \sum_{j=0}^{2k} a_j z^j, \qquad g(z) = \sum_{j=0}^k b_j z^j \qquad\text{and}\qquad h(z) = \sum_{j=0}^k c_j z^j\]
	and note that $|b_0| = t/|c_0|=t/C$. By the Cauchy product formula we find that
	\begin{equation} \label{eq:cauchy}
		a_k = \sum_{j=0}^k b_j c_{k-j} = t \frac{c_k}{C}\frac{b_0}{|b_0|} + \sum_{j=1}^k b_j c_{k-j}.
	\end{equation}
	Suppose that $\widetilde{g} \in H^2$ satisfies $|\widetilde{g}(0)|=t/C$ and $\|\widetilde{g}\|_{H^2}\leq1$. Define $\widetilde{f}=\widetilde{g}h$. The Cauchy--Schwarz inequality shows that $\|\widetilde{f}\|_{H^1}\leq1$, so the extremality of $f$ implies that $|\widetilde{a}_k|\leq|a_k|$. Inspecting \eqref{eq:cauchy} and using the Cauchy--Schwarz inequality, we find that the optimal $g$ must therefore satisfy
	\begin{equation} \label{eq:goptimal}
		g(z) = \frac{t}{C}\frac{\overline{c_k}}{|c_k|} + \sqrt{\frac{1-t^2/C^2}{1-|c_k|^2}} \sum_{j=1}^k \overline{c_{k-j}}z^j,
	\end{equation}
	where we used that $\|h\|_{H^2}=1$. Using that $c_0=C$, we compare the coefficients for $z^k$ in \eqref{eq:goptimal} with the definition of $g$, to find that
	\[\sqrt{\frac{1-t^2/C^2}{1-|c_k|^2}} C = C \prod_{j=l+1}^{k} \overline{\alpha_j} \qquad \implies \qquad \frac{1-t^2/C^2}{1-|c_k|^2} = B_l^2.\]
	Next we insert $t=C^2 A_l$ from the definition of $f=gh$ and $|c_k|^2 = C^2 A_l^2 B_l^2$ from the definition of $h$ to obtain
	\begin{equation} \label{eq:gzk}
		\frac{1-C^2 A_l^2}{1-C^2 A_l^2 B_l^2} = B_l^2 \qquad \Longleftrightarrow \qquad \frac{(1-B_l^2)(1-C^2 A_l^2(1+B_l^2))}{1-C^2 A_l^2 B_l^2}=0.
	\end{equation}
	The additional information we require is encoded in the equation on the right hand side of \eqref{eq:gzk}.
	
	Suppose that $l\geq1$. Evidently $A_l<1$, since $|\alpha_j|<1$ for $j=1,\ldots,l$ by Lemma~\ref{lem:structure}. It follows that the second factor on the right hand side of \eqref{eq:gzk} can never be $0$, since the trivial estimate \eqref{eq:trivialest} implies that
	\begin{equation} \label{eq:C2est}
		C^2 \leq \frac{1}{1+A_l^2 B_l^2}<\frac{1}{A_l^2(1+B_l^2)}.
	\end{equation}
	From the right hand side of \eqref{eq:gzk} we thus find that $B_l=1$, which shows that $C^2<1/(2 A_l^2)$ by \eqref{eq:C2est}. Since $t = C^2 A_l$, we conclude that $0 \leq t <1/2$. 
	
	By the contrapositive, we have established that if $1/2 \leq t \leq 1$, then the extremal for $\Phi_k(1,t)$ has $l=0$. In this case $A_0=1$ by definition, which shows that $C=\sqrt{t}$. The right hand side of \eqref{eq:gzk} becomes
	\[\frac{(1-B_0^2)(1-t(1+B_0^2))}{1-tB_0^2}=0,\]
	so either $B_0=1$ or $B_0^2 = 1/t-1$. Returning to the definition of $h$ we find that $|c_0|^2 = t$ and $|c_k|^2 = t B_0^2$. Consequently,
	\[1 = \|h\|_{H^2}^2 = t(1+B_0^2)+\sum_{j=1}^{k-1}|c_j|^2.\]
	Since $1/2 \leq t \leq1$, we find that both $B_0=1$ and $B_0^2=1/t-1$ will imply that $c_j=0$ for $j=1,\ldots,k-1$. Thus $h(z)=\sqrt{t}+\sqrt{1-t}\,z^k$. When $l=0$ we have $g=h$, which shows that the unique extremal is
	\[f(z)=\big(\sqrt{t}+\sqrt{1-t}\,z^k\big)^2,\]
	which is of the form $f_k(z)=f_1(z^k)$ as claimed.
\end{proof}

\begin{proof}[Proof of Theorem \ref{thm:unique} \normalfont{(c)}]
	In the case $0\leq t < 1/2$, we know from Theorem~\ref{thm:BK1} and Theorem \ref{thm:BK2} that $\Phi_k(1,t)=1$. See also Figure~\ref{fig:Phi1}. The stated claim follows from Exercise~3 on page~143 of \cite{Duren} by scaling and rotating the function
	\[f(z) = C \prod_{j=1}^k \big(\lambda_j-z\big)\big(1-\overline{\lambda_j}z\big)\]
	to satisfy the conditions $\|f\|_{H^1}=1$, $f(0)>0$ and $f^{(k)}(0)>0$. If the resulting function satisfies $f(0)=t$, then it is an extremal for $\Phi_k(p,t)$ and every extremal is obtained in this way. (This can be established similarly to the case (b) above.)
\end{proof}

\section{The extremal problem $\Phi_k(p,t)$ for $k\geq2$ and $0<p<1$} \label{sec:Phi201}
The purpose of this final section is to record some observations pertaining to the extremal problem \eqref{eq:Phi} in the unresolved case $k\geq2$ and $0<p<1$. 

Suppose that $k\geq0$ and consider the related extremal problem
\[\Psi_k(p) = \sup \left\{\mre\frac{f^{(k)}(0)}{k!}\,:\, \|f\|_{H^p}\leq 1 \right\}.\]
Evidently, $\Psi_0(p)=1$ for every $0<p\leq\infty$ and the unique extremal is $f(z)=1$. Recall (from \cite{BS20} or \cite{Kabaila60}) that the extremals for $\Psi_k$ satisfy a structure result identical to Lemma~\ref{lem:structure}. Note that the parameter $l$ in Lemma~\ref{lem:structure} describes the number of zeroes of the extremal in $\mathbb{D}$. Conjecture~1 from \cite[Sec.~5]{BS20} states that the extremal for $\Psi_k(p)$ does not vanish in $\mathbb{D}$ when $0<p<1$. The conjecture has been verified in the cases $k=0,1,2$ and for $(k,p)=(3,2/3)$.

Let us now suppose that $k\geq1$. There are two obvious connections between the extremal problems $\Phi_k$ and $\Psi_k$. Namely, 
\[\Phi_k(p,0)=\Psi_{k-1}(p) \qquad \text{and} \qquad \max_{0 \leq t \leq 1} \Phi_k(p,t) = \Psi_k(p).\]
Assume that the above-mentioned conjecture from \cite{BS20} holds. This assumption yields that the extremal for $\Phi_k(p,0)$ has precisely one zero in $\mathbb{D}$ and the extremal for the $t$ which maximizes $\Phi_k(p,t)$ does not vanish in $\mathbb{D}$. Note that the extremal for $\Phi_k(p,1)$, which is $f(z)=1$, does not vanish in $\mathbb{D}$.

\begin{question}
	Suppose that $0<p<1$. Is it true that the extremal for $\Phi_k(p,t)$ has at most one zero in $\mathbb{D}$?
\end{question}

We have verified numerically that the question has an affirmative answer for $k=2$. Note that for $1<p\leq\infty$, the extremal for $\Phi_k(p,t)$ either has $0$ or $k$ zeroes in $\mathbb{D}$ by Theorem~\ref{thm:unique} (a). In the case $p=1$, the extremal may have anywhere from $0$ to $k$ zeroes by Theorem~\ref{thm:unique} (b) and (c).

As mentioned in the introduction, Theorem~\ref{thm:fwiener} yields the estimates
\[\Phi_1(p,t) \leq \Phi_k(p,t) \leq k^{1/p-1} \Phi_1(p,t).\]
The upper bound is only attained if $\Phi_1(p,t)=0$ which happens if and only if $t=1$. Of course, since $\Phi_1(p,1)=0$ the lower bound is also attained. 

\begin{question}
	Fix $k\geq2$ and $0<p<1$. Is there some $t_0$ such that $\Phi_k(p,t)=\Phi_1(p,t)$ holds for every $t_0 \leq t \leq 1$?
\end{question}

By a combination of numerical and analytical computations, we have strong evidence that the question has an affirmative answer for $k=2$ and that in this case
\[t_0 = \left(1+\left(\frac{p}{2-p}\right)^2\right)^{1/p}.\]
Let us close by briefly explaining our reasoning. We began by considering the case case $l=0$ in Lemma~\ref{lem:structure}. Setting 
\[\widetilde{f}=\widetilde{g} h^{2/p-1}\]
and arguing as in the proof of Theorem~\ref{thm:unique} (b) (see also \cite{BS20}), we found if $t\geq t_0$, then the only possible extremal for $\Phi_2(p,t)$ with $l=0$ is of the form $f_2(z)=f_1(z^2)$ where $f_1$ is the corresponding extremal for $\Phi_1(p,t)$. Next, if $l=2$ then (as in the case $k=1$) we can only obtain $t$-values in the range $0 \leq t \leq 2^{-1/p} \sqrt{p}(2-p)^{1/p-1/2}$. However, since
\[2^{-1/p} \sqrt{p}(2-p)^{1/p-1/2} < t_0\]
for $0<p<1$ we can ignore the case $l=2$. The case $l=1$ was excluded by numerical computations.

\appendix
\section{Proof of Lemma~\ref{lem:equation}} \label{app:equation}
We will frequently appeal to the following corollary of Rolle's theorem: Suppose that $f$ is continuously differentiable on $[a,b]$ and that $f'(x)=0$ has precisely $n$ solutions on $(a,b)$. Then $f(x)=0$ can have at most $n+1$ solutions on $[a,b]$.

We are interested in solutions of the equation $F_p(\alpha)=0$ on the interval $(0,1)$, where we recall from \eqref{eq:Fpa} that 
\[F_p(\alpha) = p^2 \alpha^{-2}+2p(2-p) + (2-p)^2 \alpha^2 - 4\left(\alpha^{-p}+\alpha^{2-p}-1\right).\]

The initial step in the proof of Lemma~\ref{lem:equation} is to identify the critical points of $F_p$ on the interval $0<\alpha<1$. It turns out that there is only one.
\begin{lemma} \label{lem:prelim1}
	Fix $0<p<1$ and let $F_p$ be as in \eqref{eq:Fpa}. The equation $F_p'(\alpha)=0$ has the unique solution
	\[\alpha = \alpha_2 = \sqrt{\frac{p}{2-p}}\]
	on $0<\alpha<1$.
\end{lemma}

\begin{proof}
	We begin by computing
	\[F_p'(\alpha) = -2p^2 \alpha^{-3}+2(2-p)^2\alpha +4p \alpha^{-p-1}-4(2-p)\alpha^{1-p}.\]
	The solutions of the equation $F_p'(\alpha)=0$ on $0<\alpha<1$ do not change if we multiply both sides by $\alpha^{1+p}/(4-2p)$. Hence, we consider the equation $G_p(\alpha)=0$, where
	\[G_p(\alpha) = \frac{\alpha^{1+p}}{2(2-p)} F_p'(\alpha) = -\frac{p^2}{2-p} \alpha^{p-2} + (2-p) \alpha^{2+p}+\frac{2p}{2-p}-2\alpha^2.\]
	Evidently,
	\[G_p'(\alpha) = \alpha \left(p^2 \alpha^{p-4} + (4-p^2) \alpha^p-4\right),\]
	and the sign of $G_p'(\alpha)$ is the same as the sign of $p^2 \alpha^{p-4} + (4-p^2) \alpha^p-4$. Since
	\[\frac{d}{d\alpha} \left(p^2 \alpha^{p-4} + (4-p^2) \alpha^p-4\right) = 0 \qquad \Longleftrightarrow \qquad \alpha = \sqrt[4]{\frac{4p-p^2}{4-p^2}},\]
	and since $G_p'(1)=0$, we conclude that $G_p'$ changes sign at most once on $0<\alpha<1$. Since $G_p(0)= -\infty$, this means that $G_p(\alpha)=0$ can have at most two solutions on $(0,1]$. Hence $F_p'(\alpha)=0$ can have at most two solutions on $(0,1]$. It is easy to verify that these solutions are
	\[\alpha = \sqrt{\frac{p}{2-p}} \qquad \text{and} \qquad \alpha=1,\]
	and hence the proof is complete.
\end{proof}

We next want to demonstrate that $F_\alpha(\alpha_1)>0$ and $F_\alpha(\alpha_2)<0$ where $\alpha_1$ and $\alpha_2$ are from \eqref{eq:alpha1} and \eqref{eq:alpha2}, respectively. 

\begin{lemma} \label{lem:prelim2}
	Fix $0<p<1$. If $\alpha_2=\sqrt{\frac{p}{2-p}}$, then $F_p(\alpha_2)<0$.
\end{lemma}

\begin{proof}
	We begin reformulating the inequality $F_p(\alpha_2)<0$ as $H(p)>0$, for 
	\[H(p)=-\frac{2-p}{4} \alpha_2^p F_p(\alpha_2) = 2 - \left(1+2p-p^2\right) p^{p/2}(2-p)^{(2-p)/2}.\]
	Since we have $H(0)=H(1)=0$, it is sufficient to prove that the function $H$ has precisely one critical point on $0<p<1$ and that it is strictly positive for some $0<p<1$. We first check that
	\[H(1/2) = \frac{16-7\cdot 3^{3/4}}{8}>0.\]
	We then compute
	\[H'(p) = -p^{p/2}(2-p)^{(2-p)/2}\left(2(1-p) + \frac{\left(1+2p-p^2\right)}{2} \log\left(\frac{p}{2-p}\right)\right).\]
	The first factor is non-zero, so we therefore need to check that the equation $I(p)=0$ has only one solution on $0<p<1$, where
	\[I(p)= \frac{4(1-p)}{1+2p-p^2} + \log\left(\frac{p}{2-p}\right).\]
	We compute
	\[I'(p) = \frac{-4\left(3-2p+p^2\right)}{\left(1+2p-p^2\right)^2} + \frac{2}{p(2-p)} = \frac{2(1-p)^2 \left(3p^2-6p+1\right)}{p(2-p)\left(1+2p-p^2\right)^2}.\]
	Hence $I'(p)=0$ has the unique solution $p_0=1-\sqrt{2/3}$ on the interval $0<p<1$. Noting that $I(0)=-\infty$ and $I(1)=0$, we conclude by verifying that
	\[I(p_0) = \sqrt{6}+\log\big(5-2\sqrt{6}\,\big)>0\]
	which demonstrates that $I(p)=0$ has a unique solution on $0<p<1$.
\end{proof}

\begin{lemma} \label{lem:prelim3}
	Fix $0<p<1$. Let $\alpha_1$ denote the unique solution of the equation $1-2\alpha^p+\alpha^2=0$ on the interval $(0,1)$. Then $F_p(\alpha_1)>0$.
\end{lemma}

\begin{proof}
	Using the equation defining $\alpha_1$, we see that $\alpha_1^{-p}+\alpha_1^{2-p}-1=1$. Hence,
	\begin{align*}
		F_p(\alpha_1) &= \frac{p^2}{\alpha_1^2}+2p(2-p)+(2-p)^2\alpha_1^2-4 \\ 
		&= \left(\frac{p}{\alpha_1}+\alpha_1(2-p)+2\right)\left(\frac{1}{\alpha_1}-1\right)\left(p-\alpha_1(2-p)\right).
	\end{align*}
	The first two factors are strictly positive for every $0<\alpha_1<1$ and every $0<p<1$. Consequently, $F_p(\alpha_1)>0$ if and only if $\alpha_1<p/(2-p)$. The function 
	\[J_p(\alpha)=1-2\alpha^p+\alpha^2\]
	satisfies $J_p(0)=1$ and $J_p(1)=0$. Moreover, $J_p$ is strictly decreasing on $(0,p^{2-p})$ and strictly increasing on $(p^{2-p},1)$. Since $\alpha_1$ is the unique solution to $J_p(\alpha)=0$ for $0<\alpha<1$, the desired inequality $\alpha_1<p/(2-p)$ is equivalent to
	\[0>J_p\left(\frac{p}{2-p}\right)=1-2\left(\frac{p}{2-p}\right)^p + \left(\frac{p}{2-p}\right)^2.\]
	In order to establish this inequality, we multiply by $(2-p)^2/2$ on both sides to get the equivalent inequality $K(p)<0$, where
	\[K(p) = 2-2p+p^2 - p^p (2-p)^{2-p}.\]
	Our plan is to use Taylor's theorem to write
	\[K(p) = K(1)+K'(1)(p-1)+\frac{K''(\eta)}{2}\,(p-1)^2\]
	where $0<p<\eta<1$. The claim will follow if we can prove that $K(1)=K'(1)=0$ and $K''(p)<0$ for $0<p<1$. Hence we compute
	\begin{align*}
		K'(p) &= -2+2p - p^p (2-p)^{2-p} \log\left(\frac{p}{2-p}\right), \\
		K''(p) &= 2 - p^p(2-p)^{2-p} \left(\log^2\left(\frac{p}{2-p}\right)+\frac{2}{p(2-p)}\right).
	\end{align*}
	Evidently, $K(1)=K'(1)=K''(1)=0$. Hence we are done if we can prove that $K''$ is strictly increasing on $0<p<1$. This will follow once we verify that both 
	\[p^p (2-p)^{2-p} \qquad \text{and} \qquad \log^2\left(\frac{p}{2-p}\right)+\frac{2}{p(2-p)}\]
	are strictly positive and strictly decreasing on $0<p<1$. Strict positivity is obvious. The first function is strictly decreasing since
	\[\frac{d}{dp} \left( p^p (2-p)^{2-p} \right)= p^p (2-p)^{2-p} \log\left(\frac{p}{2-p}\right) \]
	and $\log(p/(2-p))<0$ for $0<p<1$. For the second function, we check that
	\[\frac{d}{dp}\left(\log^2\left(\frac{p}{2-p}\right)+\frac{2}{p(2-p)}\right) = \frac{4}{p^2}\left(\frac{p}{2-p}\log\left(\frac{p}{2-p}\right)+\frac{p-1}{(2-p)^2}\right)<0,\]
	where for the final inequality we have again used that $\log(p/(2-p))<0$.
\end{proof}

We can finally wrap up the proof of Lemma~\ref{lem:equation}.

\begin{proof}[Proof of Lemma~\ref{lem:equation}]
	By Lemma~\ref{lem:prelim1} we know that $F_p'(\alpha)=0$ has precisely one solution for $0<\alpha<1$. Since $F_p(0)=\infty$ and $F_p(1)=0$, this implies that the equation $F_p(\alpha)=0$ can have at most one solution on the interval $(0,1)$. Lemma~\ref{lem:prelim2} shows that there is exactly one solution, since $F_p(\alpha_2)<0$. Let $\alpha_p$ denote this solution. Inspecting the endpoints again, we find that $F_p(\alpha)>0$ for $0<\alpha<\alpha_p$ and $F_p(\alpha)<0$ for $\alpha_p<\alpha<1$. Using Lemma~\ref{lem:prelim2} again we conclude that $\alpha_p<\alpha_2$, while the inequality $\alpha_1<\alpha_p$ follows similarly from Lemma~\ref{lem:prelim3}.
\end{proof}

\bibliographystyle{amsplain} 
\bibliography{1extremal}

\end{document}